   \def\obs#1{{\bf (*** #1 ***)} }

 \def\obs#1{}     
\documentclass[twoside,letterpaper,draft,11pt]{amsart}

\usepackage{amsmath,amsthm,latexsym,xspace,amscd,amssymb,xypic, amscd,amssymb}               
\usepackage{color}
\usepackage{fullpage}

\newtheorem{teo}{Theorem}[section]
\newtheorem{defi}[teo]{Definition}
\newtheorem{lema}[teo]{Lemma}
\newtheorem{cor}[teo]{Corollary}
\newtheorem{prop}[teo]{Proposition}

\newtheorem{exe}[teo]{Example}
\newtheorem{rem}[teo]{Remark}

\newcommand{\X}{{\mathbb X}}
\newcommand{\Y}{{\mathbb Y}}

\newcommand{\R}{{\mathbb R}}

\newcommand{\ex}{{\exists}}
\usepackage[usenames,dvipsnames]{xcolor}
\newcommand{\m}{{}^{-1}}
\newcommand{\mt}{\mapsto}

\def\ndv{\ {\mid \kern -0.7 em {\scriptstyle \not}} \ \ }
\def\nd{\ {\mid \kern -0.4 em {\scriptstyle \not}} \ \ }

\newcommand{\N}{{\mathbb N}}

\numberwithin{equation}{section}
\title{The open mapping principle for partial actions of Polish groups}

\begin{document}
\author[J.\ G\'omez ]{J. G\'omez }
\address{Escuela de Matem\'aticas, Universidad Industrial de Santander, Cra. 27 Calle 9 UIS Edificio 45\\  Bucaramanga, Colombia}\email{jorge.gomez14@correo.uis.edu.co }

\author[H.\ Pinedo ]{H. Pinedo }
\address{Escuela de Matem\'aticas, Universidad Industrial de Santander, Cra. 27 Calle 9 UIS Edificio 45\\  Bucaramanga, Colombia}\email{hpinedot@uis.edu.co }
\author[C.\ Uzcategui ]{C. Uzc\'ategui }
\address{Escuela de Matem\'aticas, Universidad Industrial de Santander, Cra. 27 Calle 9 UIS Edificio 45\\  Bucaramanga, Colombia}\email{cuzcatea@uis.edu.co}
\date{\today}
\thanks{The authors thank La Vicerrector\'ia de Investigaci\'on y Extensi\'on de la Universidad Industrial de Santander for the financial support for this work,  which is part  of the VIE project \# 5761.}

\begin{abstract}  
We present a extension of the classical open mapping principle and Effros' theorem  for Polish group actions to the context of partial group actions. 
	
\end{abstract}
\maketitle
\section{Introduction}

Let $a:G\times \X\rightarrow \X$ be a continuous  action of a Polish group $G$ on a non meager,  separable metrizable space $\X$. The open mapping principle states that, if $a$ is transitive, then the evaluation map $g\mapsto a(g,x)$ is open, for each $x\in\X$. The first proof of this result is due to Effros \cite{EF}, since then several other proofs has been found \cite{AN, GA, HO,VM}.   The original statement of Effros' theorem says  that for a Polish space $\X$, the orbit of  point $G\cdot x$ is Polish iff  $\X$ is homeomorphic to the coset space $G/G_x$ (where $G_x$ is the stabilizer of $x$).  This result is fundamental for the classification of orbit equivalence relations \cite{GA}. Moreover, motivated by  Effros' theorem,  the question of which Polish spaces admit a transitive action of a Polish group (and therefore are homeomorphic to a coset space)  has been investigated   \cite{vanMill2008}.  Effros' results turned out to have an important influence  upon the development of the theory of homogeneous continua \cite{CM}. Generalization of  Effros' theorem are  discussed in \cite{LU,Ostaszew2015}.

Given an action $a:G\times \Y\rightarrow \Y$ of a group $G$ over a set $\Y$ and an invariant  subset $\mathbb{X}$ of $\Y$ (i.e. $a(g,x)\in \mathbb{X},$ for all $x\in \mathbb{X}$ and $g\in G$),  the restriction of $a$ to $G\times\mathbb{X}$ is an action of $G$ over $\mathbb{X}$. However, if $\mathbb{X}$ is not invariant,  we get what is called a {\em partial action} on $\mathbb{X}$: a collection of partial maps $\{m_g\}_{g\in G}$ on  $\mathbb{X}$ satisfying $m_1={\rm id}_\mathbb{X}$ and $m_g\circ m_h\subseteq m_{gh},$ for all $g,h\in G$.
The notion of partial action of a group    is a weakening of classical group actions and was introduced by  R. Exel in \cite{E1} motivated by  problems arising from $C^*$-algebras, in particular, they have been useful  to endow relevant classes of $C$*-algebras
with a general structure of a partial crossed product (see for instance \cite{ABMP, EV}). Recently, partial actions of groups have been considered in many different contexts, \cite[ p. 89]{KL}, moreover,  they have been an efficient tool to develop a new cohomological theory \cite{DKh, P}.    

In the topological context, partial actions were introduced in \cite{AB,KL}. They consist of a family of homeomorphism between a pair of open subsets of the  spaces.  A natural question is whether  a partial action  of a group $G$ can be realized as a restriction of a  global action of $G$. This problem was studied  by Abadie \cite{AB} and independently  by J. Kellendonk and M. Lawson  \cite{KL}.  They showed that for any continuous partial action $m$ of a topological group $G$ on a topological space $\mathbb{X}$, there is a topological space $\Y$ and a continuous action $a$ of $G$ on  $\Y$ such that $\mathbb{X}$ is a subspace of $\Y$ and $m$ is the restriction of $a$ to $\mathbb{X}$.  Such a space $\Y$ is called a globalization of $\mathbb{X}$. They also show that there is a minimal globalization $\mathbb{X}_G$ called the enveloping space of $\mathbb{X}$.
Partial actions of Polish groups on Polish spaces were studied in  \cite{GPU1,PU,PU2}. In particular,  sufficient conditions for the space $\X_G$ to be Polish were found (see  Theorem \ref{xgpol} below).

The main purpose of this paper is to show the open mapping principle  and Effros' theorem in the setting of partial actions of Polish groups  (see Theorems  \ref{Teo.m_x.abierta}, \ref{ef}  and \ref{ef2}).  It is worth to notice that in the case the space  $\X_G$ is  metrizable, the open mapping principle for partial actions can be obtained as a corollary  of the corresponding result for  global actions (see Remark \ref{glometric}).

\section{Preliminaries}
Throughout this work $G$ will be a topological  group  with identity $1,$ $\X$ a topological space unless it is said otherwise.  A Polish space  is a topological space which is is separable and completely metrizable, and a Polish group is a topological group whose topology is Polish. We use \cite{GA,KE} as a general reference for descriptive set theoretic notions and results. 

We recall the definition of a partial action. Let   $m\colon  G\times \X\to \X,\,\,(g,x)\mt m(g,x)=g\cdot x\in \X,$ be a partially defined function.
As usual, we write $\ex g\cdot x$ to mean that $(g,x)$ is in the domain of $m.$ Then $m$  is called a (set theoretic) {\it partial action} of $G$ on $\X,$ if for all $g,h\in G$ and $x\in \X$ we have:
\smallskip

\noindent (PA1) $\ex g\cdot x$ implies $\ex g\m\cdot (g\cdot x)$ and $g\m\cdot (g\cdot x)=x,$
\smallskip

\noindent (PA2)  $\ex g \cdot (h\cdot x)$ implies $\ex (g h)\cdot x$ and  $g \cdot (h\cdot x)=(g h)\cdot x,$
\smallskip

\noindent (PA3) $\ex 1\cdot x,$ and $1\cdot x=x$.
 \smallskip

Let $x\in\X.$ Then  $G^x\cdot x=\{g\cdot x\mid g\in G^x\}$  and $G_x=\{g\in G^x\mid g\cdot x=x\},$ are {\it orbit} and  the {\it stabilizer} of $x,$ respectively. 
We say that a partial action $m$ is {\it transitive}, if given $x,y\in \X$ there exists $g\in G$ such that $x\in \X_{g\m}$ and $g\cdot x=y$. Equivalently, $m$ is transitive, if and only if, $G^x\cdot x=\X,$ for all $x\in \X.$

The following  lemma is   clear.
\begin{lema}\label{gxg} Let $m$ be  a partial action of  $G$ in $\X$ and  $(g,x)\in G*\X.$ Then
\begin{itemize}
\item[(i)] For  $h\in G,$ if   $g^{-1}h\in G_x$ then $h\in G^x.$
\item[(ii)] The map  $r_{g\m}:G\ni h\mapsto hg\m\in   G$  is a homeomorphism with $r_{g\m}(G^x)=G^{g\cdot x}.$ In particular, if $m$ is transitive   and $G^x$ is open,  then $G^y$ is open, for all $y\in \X.$ 
\end{itemize}

\end{lema}

 \smallskip

We consider $G\times \X$ with the product topology. The domain of the partial action is the set $G* \X=\{(g,x)\in G\times \X\mid \ex g\cdot x \}$ endowed  with  the induced topology. For $g\in G$ and $x\in \X$ write  $\X_{g\m}=\{ x\in \X\mid \ex g\cdot x\},$ and $G^x=\{g\in G\mid \ex g\cdot x\}.$    Then we obtain the  family of maps $m_g\colon \X_{g\m}\ni x\mt g\cdot x\in\X_g,$ for all $g\in G.$
By \cite[Lemma 1.2]{QR} a  partial action $m\colon G*\X\to\X$  can equivalently be formulated in terms  of a family of bijections $m=\{m_g\colon \X_{g\m}\to \X_g\}_{g\in G}$ in the following sense.

\begin{prop}  \label{fam} A partial action $m$ of $G$ on $\X$ is a family $m=\{m_g\colon \X_{g\m}\to\X_g\}_{g\in G},$ where $\X_g\subseteq \X,$  $m_g\colon\X_{g\m}\to \X_g$ is bijective, for  all $g\in G,$  and such that:
\begin{itemize}
\item[(i)]$\X_1=\X$ and $m_1=\rm{id}_\X;$
\item[(ii)]  $m_g( \X_{g\m}\cap \X_h)=\X_g\cap \X_{gh};$
\item[(iii)] $m_gm_h\colon \X_{h\m}\cap  \X_{ h\m g\m}\to \X_g\cap \X_{gh},$ and $m_gm_h=m_{gh}$ in $ \X_{h\m}\cap  \X_{g\m h\m};$
\end{itemize}
for all $g,h\in G.$
\end{prop}
 
\begin{defi}\label{topar}
A topological partial action of  $G$ on  $\X$ is a partial action $m=\{m_g\colon \X_{g\m}\to\X_g\}_{g\in G}$ on the underlying set $\X$, such that each $\X_g$ is open in $\X$, and each $m_g$ is a homeomorphism. A {\em continuous partial action} is a topological partial action $m\colon G*\X\to \X$ which is  continuous.
\end{defi}

\smallskip

Now we recall  the definition of the enveloping action in the topological sense.
Let $m$ be a topological partial action of $G$ on $\X.$ Define the following
equivalence relation on $ G\times \X$:
\begin{equation*}
\label{eqgl}
(g,x)R  (h,y) \Longleftrightarrow x\in \X_{g\m h}\,\,\,\, \text{and}\, \,\,\, m_{h\m g}(x)=y,
\end{equation*}
and denote by  $[g,x]$   the equivalence class of the pair $(g,x).$
The {\it enveloping space} of  $\X$ is the set $\X_G=(G\times \X)/R$  endowed with the quotient topology. Let
\begin{equation}\label{que}
q\colon G\times\X\ni (g,x)\to[g,x]\in \X_G
\end{equation} 
 be the quotient map.  An action of $G$ on $\X_G$ is defined by   
\begin{equation}
\label{action}
\mu \colon G\times \X_G\ni (g,[h,x])\to [gh,x]\in \X_G.
\end{equation}
The  map
\begin{equation}
\label{iota}
\iota \colon \X\ni x\mt [1,x]\in \X_G
\end{equation}
induces a morphism  $\iota\colon m\to \mu$  in the category  of topological partial actions of $G$ (see \cite[page 17]{AB}). Next theorem collects some facts that will be used in the sequel. 

\begin{teo}
\label{Abadie} \cite[Theorem 1.1]{AB} and  \cite[Theorem 3.9]{KL}. Let $m$ be a continuous partial action of $G$ on $\mathbb{X}$. Then
\begin{enumerate}
		\item $q$ is continuous and open. 
		\item $\mu$ is a continuous action of $G$ on $\X_G$. 
		\item  $\iota : \mathbb{X}\rightarrow  \iota(\mathbb{X})$ is a homeomorphism.
		\item If $G*\X$ is open, then $\iota(\X)$ is open.
\end{enumerate}
\end{teo}

In general, the space $\X_G$ is not metrizable. Next result gives sufficient conditions to get the metrizability of $\X_G$. 

\begin{teo}\label{xgpol} \cite[Theorem 4.7]{PU} Let $m$ be a continuous partial action of a separable metrizable group $G$ on a separable metrizable space  $\mathbb{X}$, then $\mathbb{X}_G$ is metrizable under any of the following conditions:
	$G*\mathbb{X}$ is clopen or $\mathbb{X}$ is locally compact and $R$ is closed in $(G\times \X)^2$. If in addition,  $G$ and $\mathbb{X}$ are Polish, then  $\mathbb{X}_G$ is Polish.
	
\end{teo}
We finish this section with a result that will be needed later. 

\begin{lema} 
\label{estabilizador} Let $m$ be a  partial action of a metrizable group $G$ over a space  $\X$. 
Then the stabilizer $G_x$ of $x$ respect to  $m$ is equal to the stabilizer of $\iota(x)$ respect to $\mu,$ and therefore, if $\X_G$ is $T_1$ then $G_x$ is closed.
\end{lema}

\proof Let $g\in G_x,$ then $\exists g\cdot x$ and  $g\cdot x=x,$ which implies $[1,g\cdot x]=[1,x],$ but  $[1,g\cdot x]=[g,x].$ Thus by \eqref{action} we have $\mu_g([1,x])=[1,x]$ and $g\in G_{\iota(x)}.$ Conversely, if $[g,x]=[1,x]$ then $\exists g\cdot x$ and  $g\cdot x=x,$ that is $g\in G_x.$  Finally, if $\X_G$ is $T_1,$ then $\{\iota(x)\}$ is closed in $\X_G$ and $G_{\iota(x)}=\mu\m_{\iota(x)}(\{\iota(x)\})$ is closed in $G.$
\endproof

\section{The open mapping theorem}

The purpose of this section is to prove that for a transitive and continuous partial  action $m,$ the map $m^x \colon G^x\ni g\mapsto g\cdot x\in\X $ is open, provided that $G$ is Polish and $G^x$ is open for all $x$. 

\noindent As in most  proofs of this type of results,  we will use Baire category methods. The proof will follow the ideas presented in \cite{VM}. We start by recalling  some notions we will need.    A subset $A$ of a topological space $\X$ is {\em nowhere meager} in $\X$, if every non empty relative open subset of $A$ is not meager in $\X$. 
A subset of $\X$  is {\em analytic} if it is a continuous image of a Polish space (notice that we are not assuming for this definition that $\X$ is Polish, not even metrizable).
\smallskip

The second part of the following  proposition  generalizes   \cite[Proposition 2.2]{VM}. Since we need this result for Hausdorff spaces, we include a sketch of its proof for the sake of completeness.

\begin{prop}
\label{Propinterfat} 
Let $\X$ be a Hausdorff  space and $A\subseteq \X$ a set with the Baire property.  If  $A$  is dense and  nowhere meager, then $A$ is comeager. In particular, if $\X$ is not meager and   $A$ and $B$ are analytic, dense and nowhere meager subsets of $\X$,  then $A\cap B\neq \emptyset$. 	
\end{prop}

\begin{proof}
The first claim is straightforward. For the second claim,   it suffices to show that every analytic set on a Hausdorff space  has the Baire property and then use the obvious fact  that the intersection of two comeager sets is comeager and hence non empty as $\X$ is not meager.  The proof that an analytic subset of a Hausdorff space has the Baire property  uses two facts: (i)  the  classical result saying that the collection of Baire measurable subsets of a topological space  is closed under the Suslin operation (see \cite[\S 11 VII]{KW1}) and (ii)  any analytic set is the result of applying the Suslin operation to a collection of closed sets (here we need that $\X$ is Hausdorff)	.
\end{proof}

It is a standard fact that in any first countable group there is a basis $\left\{U_n\right\}_n$ of the identity $1\in G$   with the following properties:

\begin{itemize}
	\item[(a)] $U_n$ is symmetric ($U_n=U_n^{-1}$), for each $n\in \N$ and $U_1=G;$
	\item[(b)] $U_{n+1}\subseteq U_{n+1}^2\subseteq U_n.$
\end{itemize}
We fix such a basis for the rest of the paper. 
\smallskip

Given $U\subseteq G$, $A\subseteq \X$, $g\in G$ and $x\in \X$, we denote 
$
U^x=U\cap G^x=\{h\in U : \exists h\cdot x\},
$
and 
$
g\cdot A=m_g(A\cap \X_{g^{-1}}).
$

\begin{lema}\label{LemPropAx}Let  $U, F\subseteq G$ and $x\in \X.$ Then
$\left(\bigcup_{g\in F} \left(gU^x\right)\right)^x\cdot x= \bigcup_{g\in F} (  g\cdot U^x)\cdot x.$
\end{lema}
\proof Indeed, for $y\in \left(\bigcup_{g\in F} \left(gU^x\right)\right)^x\cdot x,$ there is $g\in F$ and $u\in U^x$ such that $y=(gu)\cdot x,$ then $x\in \X_{(gu)\m}\cap \X_{u\m},$ and by (ii) of Proposition \ref{fam} we get that $u\cdot x\in \X_{g\m}\cap \X_u,$ which implies  $y=g\cdot (u\cdot x),$ thanks to (PA2) , from this we get that $y\in\bigcup_{g\in F} (  g\cdot U^x)\cdot x. $ The converse is easy.
\endproof

\begin{lema}\label{Lem1TAA} Let $m$ be a continuous partial action of $G$ on $\X.$
Let $n\in \N$ and $x\in \X$ such that $G^x$ is open. Then for every open subset  $V$ of $\X$ and $z\in V\cap U_n^x\cdot x$, there is  $m\in \N$ such that  $U_m^z\cdot z \subseteq V\cap U_n^x \cdot x.$

\end{lema}
\begin{proof} 
First of all notice that  $U_n^x=U_n\cap G^x$ is an open subset of $G$ containing the identity. Since  $z\in V\cap U_n^x\cdot x$, then $z=m^x(h)$  for some $h\in U_n^x$. Let $E=(m^x)\m(V)$, then $E$ is an  open subset of 
$G$. But  $1\in Eh^{-1}\cap U_n^xh^{-1},$ thus there is   $m\in \N$ such that $U_m\subseteq Eh^{-1}\cap U_n^xh^{-1}.$ Now we check that  $U_m^z\cdot z \subseteq V\cap U_n^x \cdot x.$ Let  $p\in U_m^z\cdot z$ and $g\in U_m^z$ such that $p=g\cdot z.$ Since $g\in U_m^z\subseteq  Eh^{-1}\cap U_n^xh^{-1},$ then $gh\in E\cap U_n^x$ and
 $$
 p=g \cdot z = g \cdot (h\cdot x)\overset{{\rm (PA2)}}=(gh)\cdot x = m^x(gh)\in m^x(E\cap U_n^x).
 $$ 
 Therefore $p\in V\cap U_n^x\cdot x,$ as desired.
\end{proof}

Next lemma is the crucial step for the proof of our main result. 

\begin{lema}\label{Prop3TAA} Let $G$ be a  Polish  group and  $m$  a   continuous transitive partial action of $G$ on a non meager Hausdorff space $\X$ such that   $G^x$ is open for all $x\in \X$. Then the following holds for all $x\in \X$.

\begin{itemize}
\item[(i)] $U_n^x\cdot x$ is not meager  for  every  $n\in\N.$

\item[(ii)] $U_n^x\cdot x$ is nowhere meager  for  every  $n\in\N.$ In particular, $\X$ is Baire. 

\item[(iii)] $\text{int}\left(\overline{U_n^x\cdot x}\right)$ is dense in $\overline{U_n^x\cdot x}$ and $x\in\text{int}\left(\overline{U_n^x\cdot x}\right)$  for  every  $n\in\N$. 
\item[(iv)] $\text{int}\left(\overline{U_{n+1}^x\cdot x}\right)\subseteq U_n^x\cdot x.$ 
\item[(v)] $x\in \text{int}\left(U_n^x\cdot x\right)$  for  every  $n\in\N.$
\item[(vi)] $U_n^x\cdot x$ is an open neighborhood of $x$  for  every  $n\in\N$.
\end{itemize}
\end{lema}

\begin{proof} 

(i) Suppose that  $U_n^x\cdot x$ is meager. Notice that  $\left\{gU_n^x\right\}_{g\in G}$ is a open cover of $G$. Also,  since $G$ is  metrizable and separable, it is  Lindel\"{o}f, then there is a countable  set  $F$  such that  $G=\displaystyle\bigcup_{g\in F} gU_n^x.$ By Lemma \ref{LemPropAx}, 
	$$
	\X=G^x\cdot x=\bigcup_{g\in F} g\cdot(U_n^x\cdot x).
	$$
For each $g\in F,$ the  set $(U_n^x\cdot x)\cap \X_{g\m}$ is meager in $\X_{g\m} $ (as  $\X_{g\m}$ is open), thus  $g\cdot (U_n^x\cdot x)=m_g(U_n^x\cdot x\cap \X_{g\m})$ is meager in $\X_g$ and hence in $\X$.  Since $F$ is countable,  $\X$ is  meager, which is a contradiction.

(ii) We shall prove  that every non empty relative open subset of  $U_n^x\cdot x$ is not meager. Let  $V$ be an open subset of $\X$ and  $z\in V\cap U_n^x\cdot x,$ by Lemma \ref{Lem1TAA}, there is $m\in\N$ such that  $U_m^z\cdot z \subseteq V\cap U_n^x\cdot x$. By  (i), $U_m^z\cdot z$ is not meager. Therefore $V\cap U_n^x\cdot x$ is not meager.  This finishes the proof that  $U_n^x\cdot x$ is nowhere meager.  Since $U_1=G$ and $\X=G^x\cdot x$, then $\X$  is nowhere meager, i.e. it is Baire. 

(iii) Let $V$ be a open subset of $\X$ such that $V\cap \overline{U_n^x\cdot x}\neq\emptyset.$ By (ii), the set $V\cap U_n^x\cdot x$ is not meager, in particular,   $\text{int}\left(\overline{V\cap U_n^x\cdot x}\right)\neq\emptyset.$ From this it follows that  $\text{int}\left(\overline{U_n^x\cdot x}\right)$ is dense in $\overline{U_n^x\cdot x}.$ 
	
To see the second claim, let  $V$ be a non empty open subset of $\X$ such that $V\subseteq \overline{U_{n+1}^x\cdot x}$. Then  there is $h\in U_{n+1}^x$ such that  $h\cdot x \in V$.   Since $U_{n+1}\subseteq U_{n+1}^2\subseteq U_n$ and $m_{h\m}$ is continuous,  then
\begin{align*}
	x\in h^{-1}\cdot (V\cap \X_h)&\subseteq h^{-1}\cdot\left( \overline{U_{n+1}^x\cdot x}\cap \X_h\right)\\
	&\subseteq \overline{h^{-1}\cdot\left(( U_{n+1}^x\cdot x)\cap \X_h\right)}\\
	&=\overline{\left\{(h^{-1}g)\cdot x:g\in U_{n+1}^x\right\}}\\
	&\subseteq \overline{U_n^x\cdot x}.
\end{align*}
Finally, since $m_{h^{-1}}$ is a  homeomorphism  and $\X_{h}$ and  $\X_{h\m}$ are open, 
	then $h^{-1}\cdot (V\cap \X_h)$ is an open subset of  $\X$ containing  $x$; thus $x\in\text{int}\left(\overline{U_n^x\cdot x}\right).$

(iv) Let  $z\in V=\text{int}\left(\overline{U_{n+1}^x\cdot x}\right)$,   $W=\text{int}\left(\overline{U_{n+1}^z\cdot z}\right)$ and  $E=V\cap W$. Thus $E$ is a open neighborhood  of $z$ and it is not meager as $\X$ is Baire (by (ii)). Clearly  $U_{n+1}^x\cdot x \cap E$ and $U_{n+1}^z\cdot z \cap E$ are dense in $E$ and  by (ii) those sets are also nowhere meager in $E$.  Since $G$ is Polish and  each $U_m^y$ is open, then $U_m^y$ is also Polish for every $y$ and $m$. Thus  $U_{n+1}^x\cdot x$ and $U_{n+1}^z\cdot z$ are analytic. In summary,  $U_{n+1}^x\cdot x\cap E$ and  $U_{n+1}^z\cdot z\cap E$ are analytic, dense and nowhere meager subsets of  $E$ (as a subspace of $\X$), then  by Proposition \ref{Propinterfat}, there exists 
$$y\in (U_{n+1}^x\cdot x\cap E) \cap (U_{n+1}^z\cdot z\cap E).$$ 
Hence, there are $g\in U_{n+1}^x$  and $h\in U_{n+1}^z$ such that  $g\cdot x=y= h\cdot z$. Let $f=h^{-1}g$, then $f\cdot x=z$. Note that $f\in U_{n+1}U_{n+1} \subseteq U_{n}$ and $f\in G^x$, then $f\in U_{n}^x$. Hence $z\in U_n^x\cdot x.$

(v)  By (iv),  $\text{int}\left(\overline{U_{n+1}^x\cdot x}\right)\subseteq U_n^x\cdot x$ and thus $\text{int}\left(\overline{U_{n+1}^x\cdot x}\right)\subseteq \text{int}\left(U_n^x\cdot x\right).$ From (iii) we have that $x\in\text{int}\left(\overline{U_{n+1}^x\cdot x}\right).$ Therefore $x\in \text{int}\left(U_n^x\cdot x\right).$

(vi) Let $z\in U_n\cdot x$. By Lemma  \ref{Lem1TAA} there  exists $m\in \N$ such that $U_m^z\cdot z \subseteq  U_n^x\cdot x$ and by (v), $z\in \text{int}(U_m^z\cdot z)$. Thus $z\in \text{int}(U_m^z\cdot z) \subseteq U_n^x\cdot x$.   Hence $U_n^x\cdot x$ is open in $\X$.
\end{proof}


Now we are ready to give  the proof of the open mapping theorem for partial actions.

\begin{teo}\label{Teo.m_x.abierta} Let $G$ be a Polish  group and $m$  a  continuous transitive  partial action of $G$ in a not meager Hausdorff space $\X.$  Suppose   $G^y$ is open in $G$ for some $y\in \X$. Then  the map $m^x:G^x\ni g \mapsto g\cdot x\in \X$ is open for every $x\in \X$. 
\end{teo}

\begin{proof} By (ii) of Lemma \ref{gxg} we assume that $G^x$ is open for all $x\in \X$. 
Let  $U^x$ be an open non empty subset of $G^x,$ where $U$ is an open subset of  $G$. We check that  $U^x\cdot x$ is an open subset of  $G^x\cdot x$. Indeed, given $z\in U^x\cdot x$, then $z=g\cdot x,$ for some $g\in U^x$. As $U$ is open,  there are $O_1$ and  $O_g$ open neighborhoods of $1$ and $g$, respectively, such that  $O_1O_g\subseteq U$ and thus $O_1^zg\subseteq U^x$. Since $\left\{U_n\right\}_n$ is a basis of neighborhoods of $1$, there exists $n\in \N$ such that  $U_n^z\subseteq O_1^z$. Note that  $z\in U_n^z\cdot z$ and
$
 U_n^z\cdot z \subseteq O_1^z \cdot z =O_1^z \cdot (g\cdot x)= (O_1^zg)\cdot x \subseteq U^x \cdot x.
$
Thus, by  Lemma \ref{Prop3TAA} (vi), the set $U_n^z\cdot z$ is open  and hence $U^x \cdot x$ is also open.
\end{proof}

\begin{exe}\label{mob}{\bf M{\"o}bius transfromations}  \cite[p. 175]{CHL} The group $G={\rm SL}(2,\mathbb{R})$ acts partially on $\mathbb{R}$ by setting 
$$
g\cdot x=\frac{ax+b}{cx+d},\,\,\,\,\,\,{\rm where}\,\,\,\,\,\, g= \left( \begin{array}{ccc}
	a & b \\
	c & d \\ \end{array} \right) \in G.
$$
\end{exe}
\noindent Notice that for all $g\in G$ the set  $\X_g=\{x\in\mathbb{R}\mid cx+d\neq 0\}$ is open and the partial action is  continuous. For $x\in \mathbb{R} $ let $t_x=\left(\begin{array}{ccc}
1 & x \\
0 & 1 \\ \end{array}\right) ,$ then for $y\in \mathbb{R}$ one has that $t_{y-x}\cdot x=y.$ 
Moreover, since
$$G^0=\left\{\left(\begin{array}{ccc}
a & b \\
c & d \\ \end{array}\right) \in G: d\neq 0\right\}$$ is open,  
then by (ii) of Lemma \ref{gxg}  we have a transitive partial action for which $G^x$ is open, for all $x\in \mathbb{R}.$ By Theorem \ref{Teo.m_x.abierta},  the map  $m^x:G^x\ni g \mapsto g\cdot x\in \mathbb{R}$ is open, for every $x\in  \mathbb{R}.$ 

\subsection{On the enveloping space  of a transitive partial action}

In this section we use   the open mapping principle to get  an improvement of Theorem \ref{xgpol} for transitive  partial actions. 
We recall a  result  that gives a (set theoretic) relation between  the enveloping space $\X_G$ and  the quotient $G/G_x,$ for $x\in \X.$
\begin{teo}
\label{global-transitiva}
\cite[Proposition 2.4, Theorem 2.6]{CHL}
Let $m$ be a transitive partial  action of a group $G$ over $\X$. Then  the enveloping action $\mu$ of $G$ over $\X_G$ is  transitive and  equivalent to the left coset action of $G$ over $G/G_x$. More precisely,  the map 
\begin{equation}\label{isophi}\phi \colon  G/G_ { x} \ni gG_{x}\to [g,x]\in \X_G
\end{equation}  is a bijection such that $\phi(h gG_x)=\mu_h\phi( gG_x),$ for any $h\in G.$
\end{teo}

The following result is straightforward. 

\begin{lema}
\label{phi-continua}
Let $G$ be a Polish  group and $m$  a transitive  continuous partial action of $G$ on a  space $\X$. Then the map  $\phi$  defined in \eqref{isophi} is continuous.
\end{lema}


Now we show that, for transitive actions,  $\X_G$ is Polish when it  is Hausdorff.

\begin{teo} 
\label{coset-space}
Let $G$ be a Polish  group and $m$  a  continuous transitive partial  action of $G$ on a not meager space $\X$ such that $G*\X$ is open.   The following statements are equivalent. 
\begin{enumerate}
\item  $\X_G$ is Hausdorff. 

\item   $\X_G$ is $T_1$ and is homeomorphic to $G/G_{x},$  for any $x\in\X$.

\item $\X_G$ is Polish. 
\end{enumerate}

\end{teo}

\begin{proof} 
(1) $\Rightarrow$ (2). Suppose that $\X_G$ is Hausdorff.  Since $\iota$ is an embedding (see Theorem \ref{Abadie}), then $\X$ is Hausdorff.  Let $x\in \X$.  We will show that $\phi$ is an homeomorphism.  By Lemma \ref{phi-continua}, we only need to show that $\phi$ is open. 
First we verify  that we can apply the open mapping Theorems  to  the  enveloping total  action $\mu$ of $G$ over $\X_G$.  In fact,  $\X_G$ is not meager in itself, as the map $q$ defined in \eqref{que} is continuous  and open and $\X$ and $G$ are not meager.  By Theorems \ref{global-transitiva} and \ref{Abadie},   $\mu$  is a  continuous transitive  action. Thus, by   Theorem \ref{Teo.m_x.abierta}, the map $\mu^{[1,x]}\colon G\to \X_G$ is open.   Now we show that  $\phi$ is open. Let $\pi\colon G\to G/G_x$ be the quotient map.  It suffices to show that if $O\subseteq G$ is open, then $\phi (\pi(O))$ is open in $\X_G$.   In fact,  since $\phi\circ \pi=\mu^{[1,x]}$ we are done. 

(2) $\Rightarrow$ (3). It is a classical result that if $H$ is a closed subgroup of Polish group $G$, then $G/H$ is also Polish (see \cite[Theorem 2.2.10]{GA}).  But by  Lemma \ref{estabilizador} we   know that $G_x$ is closed.

(3) $\Rightarrow$ (1) This part  is obvious.  \end{proof}

\begin{cor}
\label{Xpolaco}
Let $G$ be a Polish  group and $m$  a  continuous transitive partial  action of $G$ on a not meager space $\X$ such that $G*\X$ is open.  If $\X_G$ is Hausdorff, then $\X$ is Polish. 
\end{cor}

\begin{proof}
By Theorem \ref{Abadie}, $\iota(\X)$ is open and $\iota$ is an embedding, therefore $\X$ is Polish.
\end{proof}

\begin{rem}
\label{glometric} 
Suppose  $\X_G$ is  metrizable, then the  argument used in the proof of $(1) \Rightarrow (2)$ in  Theorem \ref{coset-space} shows that the open mapping theorem \ref{Teo.m_x.abierta}  follows  from the corresponding theorem for global actions. To see this, observe that $\iota \circ m^x=\mu^{[1,x]}\restriction G^x,$  and thus $\iota \circ m^x$ is an open map.  Since   $\iota$ is a continuous injection  (see Theorem \ref{Abadie}), then $m^x$ is open. 

%
%
\end{rem}

\section{On the Effros'  theorem}

In this section we extend Effros'  result \cite[Theorem 2.1]{EF} to the context of partial actions.

\begin{defi}  Let $m$ be a partial action on $\X$. The orbit equivalence relation $E_G^p$ on $\X$ is defined by
$$x E_G^p y \Longleftrightarrow \ex\, g\cdot x\,\,\,\text{and}\,\,\, g\cdot x=y ,$$ for some $g\in G.$ 
\end{defi}

The set of equivalence classes  $\X/E_G^p$  is endowed with the quotient topology. By  \cite[Lemma 3.2]{PU} the quotient map $ \X\ni x\mapsto [x] \in \X/E_G^p$ is continuous  and open, from this follows  that if $\X$ is second-countable,  so is $\X/E_G^p.$

\begin{teo}\label{ef} Let $G$ be a Polish group and $m$ be a continuous  partial action of  $G$ on the Polish space $\X.$
 Then  the following assertions are equivalent.
\begin{enumerate}
\item  $E_G^p$ is $G_\delta.$
\item $G^x\cdot x$ is $G_\delta$ in $\X,$ for every $x\in \X$.
\item $\X/E_G^p$ is $T_0.$ That is,  $ \overline{\{x\}}\neq \overline{\{y\}}$ for any $x,y\in\X/E_G^p $  with $x\ne y$.
\end{enumerate}
\end{teo}

\proof
It is  shown exactly as in the case of a global action (see for instance \cite[Theorem 3.4.4]{GA}).
Only recall that  $\X/E_G^p$   is second countable.\endproof

\begin{prop}
\label{propgx}
Let $G$ be a Polish group,  $\X$ a metric space and $m$ a topological partial action of $G$ on $\X.$ Then for a fixed $x\in\X$ the following assertions hold.
\begin{enumerate}
\item The quotient map $ \pi_p\colon G^x\ni g\to gG_x\in G^x/G_ x  $ is continuous and open.
\item  The map
\begin{equation}
\label{hati}
\hat{\iota}\colon G^x/G_ x\ni gG_x\to  gG_{\iota(x)}\in G/G_ {\iota(x)}
\end{equation}
is a topological embedding. 

\item If $G^x$ is $G_\delta$ and $G_x$ is closed {in $G$},   then $ G^x/G_ x$ is a Polish space.
\end{enumerate}
\end{prop}

\proof (1) Let $V$ be an open subset of $G^x,$ and take $W\subseteq G$ open such that $V=W\cap G^x$. Notice that  $Vg=Wg\cap G^xg$ and $G^xg=G^x,$ for all $g\in G_x.$
Then $\pi_p\m\pi_p[V]=VG_x=\bigcup_{g\in  G_x} Vg= WG_x \cap G^x$ is open in $G^x.$
\medskip

\noindent (2) It is clear that $\hat{\iota}$ is well defined and injective. Moreover,  $\hat\iota\circ \pi_p=\pi\restriction G^x,$ where $\pi$ is the canonical projection $G\to G/G_{\iota(x)},$ then $\hat{\iota} \circ \pi_p$ is continuous, which implies that $\hat{\iota}$ is continuous.
Let $Y=\hat{\iota}[G^x/G_ x]$. We will show that $\hat{\iota}:G^x/G_ x\to  Y$ is open.  Let $U\subseteq G$ open  we prove that
 $$\{gG_{\iota(x)}:\, g\in U\cap G^x\}=\pi[U]\cap Y.$$
It is clear that $\subseteq$ holds. For the other inclusion, let $h\in U$  such that $\pi(h)\in Y$. Let $g\in G^x$ such that $hG_{\iota(x)}=gG_{\iota(x)}$. Since $G_x=G_{\iota(x)}$, then
$g\m h\in G_x$ and by (i) of Lemma \ref{gxg} we have that $h\in G^x.$


\medskip

\noindent (3)   Since $G_x=G_ {\iota(x)}$ is a closed subgroup of  $G$ the space $G/G_ {\iota(x)}$ is Polish.  By (2),  $G^x/G_ x \simeq \hat \iota [G^x/G_ x]\subseteq G/G_ {\iota(x)}$, thus $G^x/G_x$ is metrizable. Finally, by (1), $G^x/G_x$ is the continuous open image of the Polish space $G^x$ (as it is  a $G_\delta$ subset of a Polish space), therefore  it is Polish by Sierpinski's theorem (see \cite[Theorem 2.2.9]{GA}). \endproof

\smallskip

Now we  show that the orbits are Borel, this is  a generalization of  a well known theorem about Polish group actions (see \cite[Proposition 3.1.10]{GA}).
\begin{cor}
Let $G$ be a Polish group, $m$  a continuous  partial action of  $G$ on the Polish space $\X$ and $x\in \X$. If   $G^x$ is $G_\delta$ and $G_x$ is closed in $G$, then   $G^x\cdot x$ is Borel.
\end{cor}

\proof By Lemma \ref{propgx}, $G^x\cdot x$ is the continuous and injective image of the Polish space  $G^x/ G_x$, then it is Borel by Lusin-Souslin's Theorem  (see \cite[Theorem 15.1]{KE}).
\endproof


\begin{teo}\label{ef2} Let $G$ be a Polish group and $m$ be a continuous partial action of  $G$ on the Polish space $\X$. Let  $x\in \X$ such that $G^x$ is $G_\delta$  and $G_x$ is closed.  Consider the following assertions.
\begin{enumerate}
\item The map 
\begin{equation*}
\phi \colon  G^x/G_ { x} \ni gG_{x}\to g\cdot  x\in G^x\cdot x
\end{equation*}
is a homeomorphism.

\item  $ G^x\cdot x$ is $G_\delta.$

\item $ G^x\cdot x$ is not meager in its relative topology.

\end{enumerate}
Then $(1)\Rightarrow (2) \Rightarrow (3)$. Moreover, if $G^x$ is open, then  all the assertions are equivalent.
\end{teo}

\proof 
Fix $x\in \X.$
Since $G^x$ is $G_\delta$   and $G_x$ is closed,  then by (3) of Proposition \ref{propgx} we have  $(1)\Rightarrow (2).$ It is clear that  $(2) \Rightarrow (3)$. Now suppose that $G^x$ is open. It is clear that $\phi$ is a  continuous bijection, so to see that 
  $(3)\Rightarrow (1)$ it suffices to show that $\phi$ is  open. Let $\Y=G^x\cdot x$ and $\overline{m}:G*\Y\to \Y$ the restriction of $m.$ Then $\overline{m}$ is a continuous transitive partial action with $\overline{m}^x=m^x.$ Since  $m^x=\phi\circ \pi_p,$  the result follows from  Proposition \ref{propgx} (1) and Theorem \ref{Teo.m_x.abierta}. 
 \endproof

 The following result is a generalization of the fact commented in the introduction. If a Polish space $\X$  admits a transitive action of a Polish group $G$, then $\X$ is homeomorphic to the coset space $G/G_x$ for any $x\in \X$.

\begin{cor}
Let $G$ be a Polish  group and $m$  a  continuous transitive partial  action of $G$ on a not meager Hausdorff space $\X$ such that $G*\X$ is open and $G_x$ is closed for some  $x\in \X$. Then $\X$ is homeomorphic to the coset space $G^x/G_x$.
\end{cor}

\proof By Corollary \ref{Xpolaco} the space $\X$ is Polish, then the results follows   by Theorem \ref{ef2}.
\endproof

Now we present  an example to illustrate  Theorems \ref{ef} and \ref{ef2}.

\begin{exe}{\bf The flow of a differentiable vector field} \cite[Example 1.2]{AB}.
Consider a smooth vector field $V \colon \X\to T\X$ on a  manifold $\X$ such that $\X$ is a Polish space. For $x\in \X$, let $\gamma_x$ be the corresponding integral curve through $x$ defined on its maximal interval $(a_x,b_x),$ that is, $\gamma_x(0)=x$ and $\gamma'_x(t)=V(\gamma_x(t))$ for all $t\in (a_x,b_x).$ For $t\in \mathbb{R}$, let  $\X_{-t}=\{x\in \X\mid t\in (a_x,b_x)\}$ and set $m_t\colon \X_{-t}\ni x\to\gamma_x(t)\in \X_t.$ Then the family $m=\{m_t\colon \X_{-t}\to \X_t\}_{t\in \mathbb{R}}$ defines a continuous  partial action of the additive group  $G=\mathbb{R}$ on $\X.$  Maximal integral curves are either constant, injective or periodic (see \cite[Exercise 9-1]{Lee}). Then for each $x\in \X,$  we have that 
$$
G_x=\left\lbrace t\in (a_x,b_x) : \gamma_x(t)=x\right\rbrace
$$
is a closed set. Additionally,  the image of $\gamma_x$ is diffeomorphic to $\R,$ $S^1,$ or $\R^0$, thus $G^x\cdot x=\gamma_x\left( (a_x,b_x)\right)$ is Polish (see \cite[Exersice 9-1 (c)]{Lee}). 
\end{exe}

The following example shows that the condition $G^x$ to be Polish in Theorem \ref{ef2}  is necessary.

\begin{exe} Consider the continuous partial action of $G=\mathbb{R}$ on itself given by $\X_g=\emptyset$ if $g\in \mathbb{R}\setminus\mathbb{Q},$  $\X_g=\mathbb{R}$ if $g\in \mathbb{Q},$ and $m_g\colon \mathbb{X}_{-g}\ni a\to g+a\in \mathbb{X}_g,$ for $g\in \mathbb{Q}.$ Then for any $x\in  \mathbb{Q}$  one has that ${G}^x=\mathbb{Q}$  which is not a Polish subset of $\mathbb{R}.$ Moreover $G_x=\{0\}$  is closed in $\mathbb{R}$ and the map 
$$
G^x/G_x \ni [r] \mapsto r+x\in G^x\cdot x
$$ 
is a homeomorphism. But  $G^x\cdot x=\mathbb{Q}$ is  meager in itself.
\end{exe}


\end{document}